\documentclass[letterpaper, 10 pt, times, conference]{ieeeconf}

\IEEEoverridecommandlockouts                              

\usepackage{graphicx} 
\usepackage{amsmath,amssymb,amsfonts}
\usepackage{mathtools}
\usepackage{commands}
\usepackage{xcolor}
\usepackage{algorithm, algpseudocode}
\usepackage{dsfont}

\usepackage{tikz}
\usetikzlibrary{arrows.meta, decorations.markings}

\usepackage{booktabs}

\usepackage{caption}
\usepackage{subcaption}

\newtheorem{thm}{Theorem}[section]
\newtheorem{prop}[thm]{Proposition}

\definecolor{natcolour}{HTML}{E57A77}

\title{\LARGE \bf A Sequential Operator-Splitting Framework for Exploration of Nonconvex Trajectory Optimization Solution Spaces}
\author{{Justin Ganiban$^{1}$, Natalia Pavlasek$^{1}$, \behcetacikmese$^{1}$}
\thanks{$^{1}$William E. Boeing Department of Aeronautics
and Astronautics, University of Washington, Seattle, WA 98195. E-mail: {\tt\small jganiban, pavlasek, behcet (at) uw (dot) edu}.}}

\begin{document}

\maketitle
\begin{abstract}
    Trajectory optimization methods provide an efficient and reliable means of computing feasible trajectories in nonconvex solution spaces. However, a well-known limitation of these algorithms is that they are inherently local in nature, and typically converge to a solution in the neighborhood of their initial guess. This paper presents a sequential operator-splitting framework, based on the alternating direction method of multipliers (ADMM), aimed at promoting exploration within the sequential convex programming (SCP) framework. In particular, diverse initial solutions are modeled as agents within the consensus ADMM framework. Driving these agents toward consensus facilitates exploration of the nonconvex optimization landscape. Numerical simulations demonstrate that the proposed method consistently yields equivalent or lower-cost solutions compared to the standard SCP approach, with the same number of or fewer agents.
\end{abstract}

\section{Introduction}

Trajectory optimization methods are widely employed in modern guidance, navigation, and control (GNC) systems for their computational efficiency and convergence guarantees~\cite{Malyuta_CSM, Bonalli_gusto}. Though GNC problems are almost always nonconvex, convex optimization-based trajectory generation methods, such as sequential convex programming (SCP), have proven to be effective in solving several real-world problems, such as autonomous drone guidance, spacecraft rendezvous and docking, and planetary landing~\cite{Malyuta_CSM, acikmese_lcvx}.
Convex optimization-based algorithms are typically guaranteed to converge to a stationary point, but since they are local methods, their solution is highly dependent on the initial guess~\cite{Elango_ctcs, bryson2018applied}.
It follows that trajectory optimization algorithms generally produce suboptimal solutions in the neighborhood of their initial guess.

Sampling-based methods present an alternative to trajectory-optimization methods. Methods such as RRT~\cite{lavalle1998_rrt}, A$^*$~\cite{hart1968_astar}, probabilistic roadmaps~\cite{kavraki2002_probroadmaps}, among others, incrementally generate discrete points by sampling from a distribution to build a solution-space exploring tree~\cite{lavalle2006planning}. Violations of constraints are detected and used to guide sampling. The exploratory nature of sampling-based methods helps mitigate the challenge posed by local minima. However, these methods generally do not scale well with problem dimension. Moreover, they do not always consider optimality, and those that do require infinitely many samples to guarantee optimality~\cite{orthey2023_sampling}.

Multi-start optimization has been used to promote exploration of a nonconvex landscape by solving multiple optimization problems with differing initializations~\cite{Martí_Resende_Ribeiro_2013, resende2016_grasp}. Diverse initial solutions are used to promote exploration of different regions of the search space. This method, while effective in combinatorial optimization applications, can lead to large computational overhead in trajectory optimization applications since each initialization explores the landscape in isolation, and information about basins of attraction is not shared, leading to independent redundant convergence to the same local minimum~\cite{denobel_2024}.

Another technique for escaping local minima within trajectory optimization is to perturb the gradient with noise, helping the optimizer escape saddle points~\cite{Kalakrishnan_Chitta_Theodorou_Pastor_Schaal_2011}. However, the injected noise can produce trajectories that violate feasibility, making the method unsuitable for safety-critical applications~\cite{Kalakrishnan_Chitta_Theodorou_Pastor_Schaal_2011}. Moreover, this method is computationally intensive due to the large number of noisy trajectories simulated.

The augmented Lagrangian method is widely used for solving constrained optimization problems via a series of unconstrained problems. The alternating direction method of multipliers (ADMM) is a variant of the augmented Lagrangian scheme widely used for distributed and convex optimization, particularly due to its ability to handle problems with a separable structure efficiently~\cite{MAL-016, Hosseini_Chapman_Mesbahi_2015, ryu2022_book}. This ability to break down a complex problem is formalized by the concept of \emph{operator splitting}, which is a mathematical framework for solving problems involving the sum of multiple terms by iteratively solving subproblems that handle each term independently.

ADMM has been used in the field of optimal control for exploiting structure in the context of model predictive control~\cite{kang2015_mpc_admm, rey2017_mpcadmm, rey2020_mpcadmm}. The problems solved by each agent are convex optimal control problems. 
In the context of nonconvex optimization, operator splitting was used to parallelize trajectory optimization for speed improvements in~\cite{Wang2021_trajsplitting}. In particular, ADMM was employed to temporally split a trajectory optimization problem so that agents could solve trajectory optimization problems with fewer discrete samples in parallel.
A similar strategy was used to simultaneously compute optimal trajectories and time-varying linear feedback control policies in~\cite{Vikas2017_Sequential_operator_splitting}.

\begin{figure*}
    \centering
    \begin{subfigure}[b]{0.49\linewidth}
    \centering
    \includegraphics[width=0.65\linewidth]{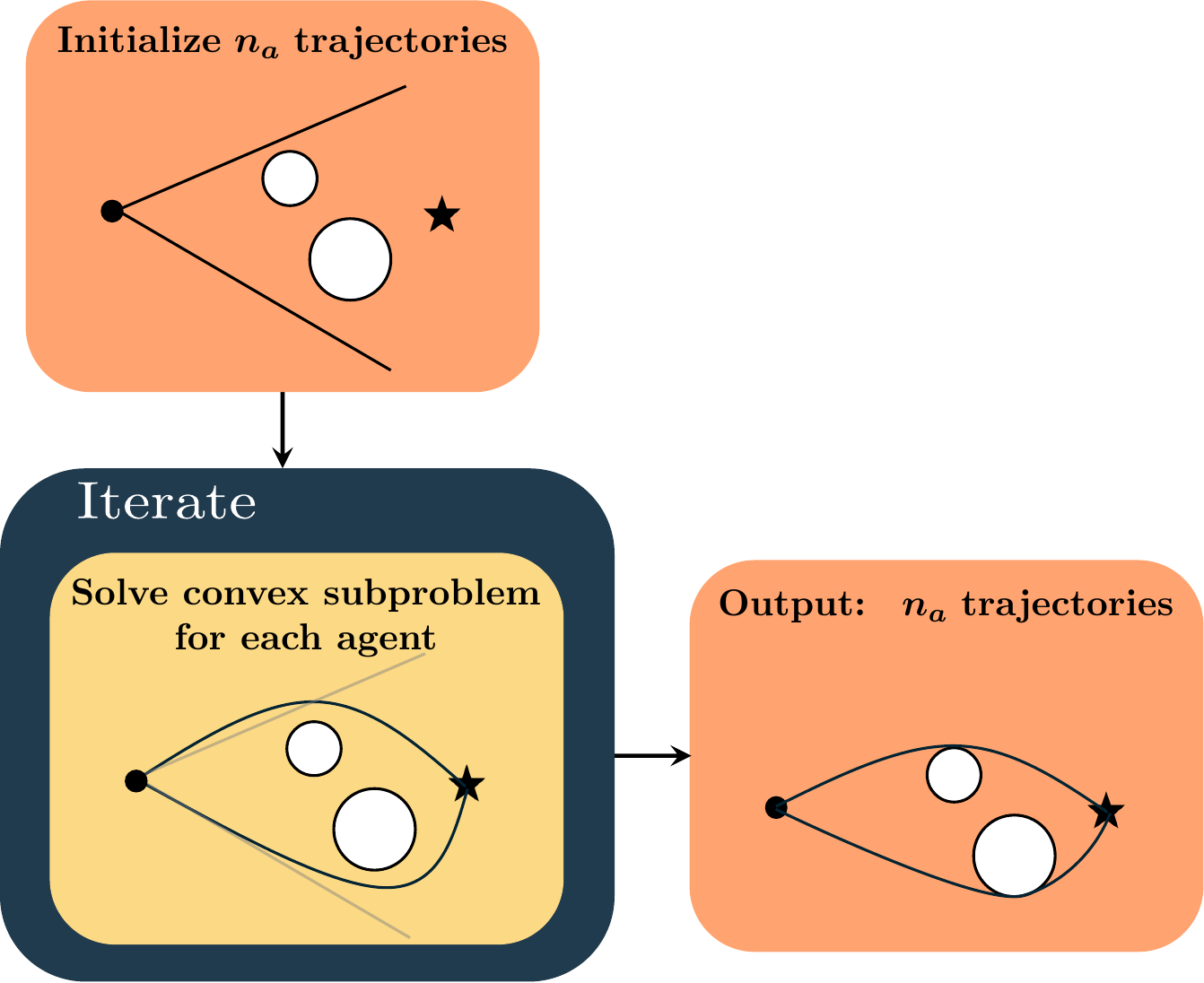}
    \caption{Standard SCP.}
    \label{fig:scp_overview}
    \end{subfigure} \hfill
    \begin{subfigure}[b]{0.49\linewidth}
    \centering
    \includegraphics[width=\linewidth]{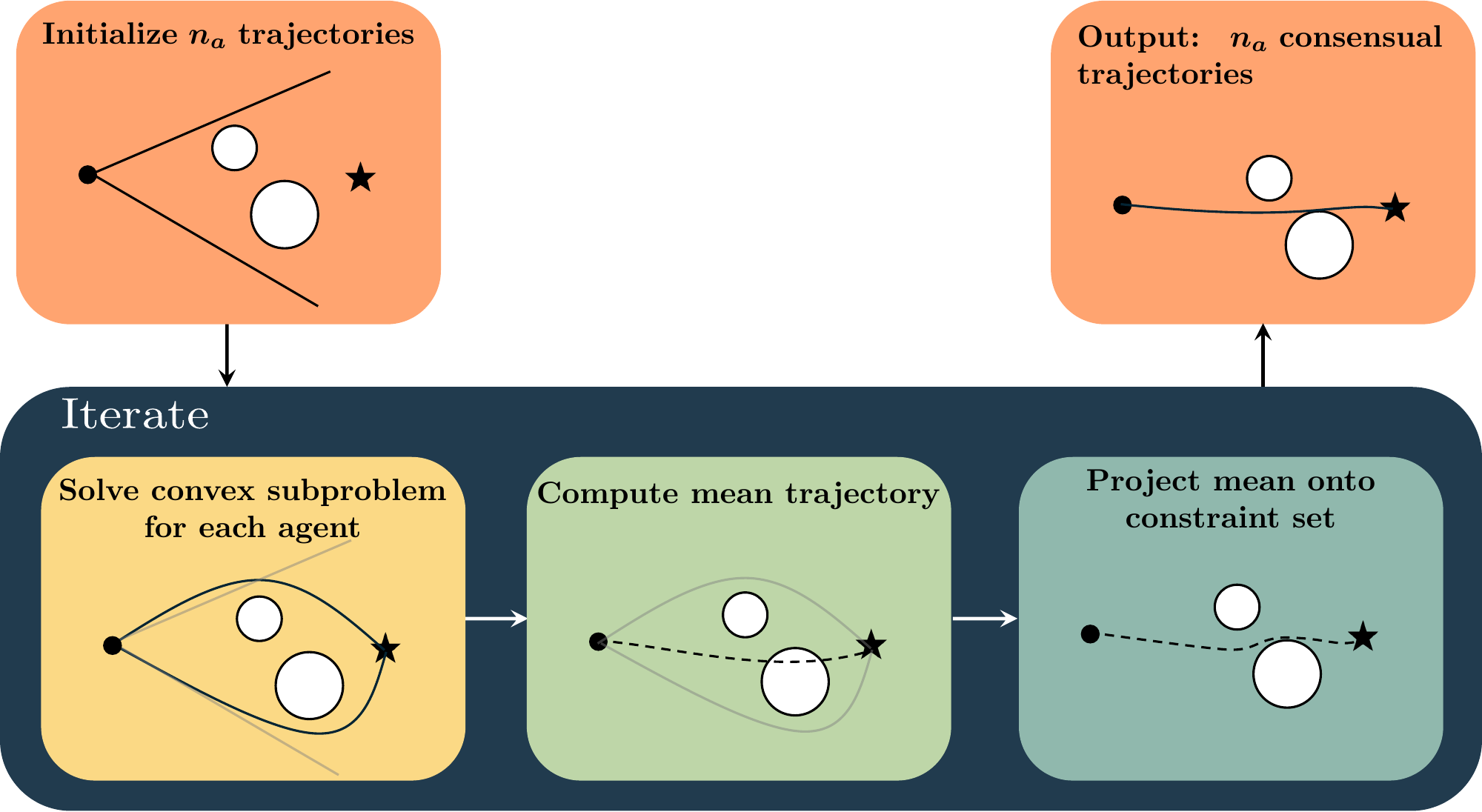}
    \caption{OS-SCP}
    \label{fig:osscp_overview}
    \end{subfigure}
    \caption{Overview of standard SCP and OS-SCP. Initial state is represented by a black circle and goal state by black star. Two obstacles are shown in white.}
    \label{fig:method_overview}
\end{figure*}

This paper leverages operator splitting to promote exploration in SCP. This is achieved by constructing a population of agents with diverse initializations that each solve for locally optimal solutions and collectively form a consensus using the ADMM framework. This provides a structured framework for promoting exploration of a shared, nonconvex solution landscape. By initializing agents with diverse initial guesses, the agents are drawn to different regions of the solution space. Then, the agents are pulled towards a consensus, allowing the agents to escape stationary points where standard gradient-based methods would stagnate. We show that this process leads to increased exploration of the solution space over multi-start methods, enabling agents to converge to solutions about which no trajectories were initialized. We demonstrate empirically that the proposed exploration-focused operator splitting method escapes local stationary points in nonconvex trajectory optimization problems where gradient-based methods fail. An overview of the standard SCP approach and the proposed Operator-Splitting SCP (OS-SCP) approach is shown in Figure~\ref{fig:method_overview}.

The remainder of this paper is structured as follows. Section~\ref{sec:background} introduces concepts used in the exploratory OS-SCP algorithm. Section~\ref{sec:algorithm} introduces the proposed algorithm, and numerical results are presented in Section~\ref{sec:results}. Finally, concluding remarks are given in Section~\ref{sec:conclusion}.

\textbf{Notation}. The following notation is used throughout this work. We denote the set of functions that are $n-$times differentiable by $\mc{C}^n$. A vector of ones of length $n$ is denoted $\mbf{1}_n$, and a vector of zeros of length $n$ is denoted $\mbf{0}_n$. We use the shorthand $[K]$ to denote the set of integers $\{0, \ldots, K\}$. The sets of real numbers are denoted by $\mbb{R}$ and positive reals by $\mbb{R}_+$, and the sets of integers and positive integers are denoted $\mbb{Z}$, and $\mbb{Z}_+$, respectively. Given a continuous-time signal $x(t)$ sampled with period $\Delta T$,
we denote the value of the signal at time instant $k\Delta T$, $k \in \mbb{Z}_+$, by $x_k = x(k \Delta T)$.

\section{Background} \label{sec:background}

This section introduces key concepts used in the proposed operator splitting SCP framework introduced in the following section.

\subsection{Consensus ADMM}
Traditionally, consensus ADMM decomposes the objective function of an optimization problem across multiple subproblems with a consensus constraint enforced through dual variables and proximal penalties. Consider the optimization problem
\begin{equation}
    \min_{\pvar}~\sum_{i} p_i(\pvar) + q(\pvar),
\end{equation}
where each $p_i : \mbb{R}^{n_\pvar}\to\mbb{R}$ is a convex term in the objective function, and $q: \mbb{R}^{n_\pvar}\to\mbb{R}$ encodes shared regularization or constraints. To apply consensus ADMM, the problem is reformulated using a splitting variable, resulting in
\begin{equation}
    \min_{\pvar_1,\ldots,z_{n_a},\bar{\pvar}}~\sum_{i}^{n_a} p_i(\pvar_i) + q(\bar{\pvar}) \quad \st \quad \pvar_i = \bar{\pvar},\quad i\in[n_a],
\end{equation}
where $\pvar_i \in \mbb{R}^{n_\pvar}$ for $i\in [n_a]$ is the state of agent $i$, and $\bar{\pvar} \in \mbb{R}^{n_\pvar}$ is the global state. Now, each $p_i(\pvar_i)$ is a convex per-agent objective, and $q(\bar{\pvar})$ represents shared regularization or constraints. The goal of consensus optimization is for each agent to eventually aggregate to a consensus such that all agents achieve the value of the global state, $\pvar_i=\bar{\pvar}$ for all $i\in[n_a]$. This is achieved through an iterative process, in which the individual agent updates are first solved in parallel over $i$, then a consensus update is performed, and finally a dual update is computed~\cite{MAL-016}, expressed as
\begin{subequations}\small \label{eq:admm}
\begin{alignat}{3}
    \pvar_i^{\itervar+1} &= \argmin_{\pvar} p_i(\pvar) + \frac{\rho}{2}\norm{\pvar -\bar{\pvar}^\itervar+\dvar_i^\itervar}^2_2,\; \label{eq:ADMM_agent_update}\\
    \bar{\pvar}^{\itervar+1}&=\argmin_{\bar{\pvar}} q(\bar{\pvar}) + \frac{\rho}{2}~\sum_{i=1}^{n_a}\norm{\pvar^{\itervar+1}_i-\bar{\pvar}+\dvar_i^\itervar}^2_2,\; &&\label{eq:ADMM_consensus_update}\\
    \dvar_i^{\itervar+1} &= \dvar_i^\itervar+(\pvar_i^{\itervar+1}-\bar{\pvar}^{\itervar+1}),\; \label{eq:ADMM_dual_update}
\end{alignat}
\end{subequations}
for $ i\in[n_a]$. Here, the primal and dual variables are represented by $\pvar_i$ and $\dvar_i\in\mbb{R}^{n_\pvar}$, respectively, $\itervar\in\mbb{Z}_+$ is the iteration variable, and $\rho \in \mbb{R}_+$ is the consensus penalty parameter.

\begin{prop}\label{prop:admm_converge}
If each $p_i$ and $q$ are convex, closed, and proper, then the consensus residuals vanish, meaning that as $\itervar \rightarrow \infty$, the primal residual 
\begin{equation}
    \norm{\delta_{r_i}^{\itervar+1}} := \norm{\pvar_i^{\itervar+1} - \bar{\pvar}^{\itervar+1}} \;\;\to\;\; 0, 
        \quad \forall i.
\end{equation}
Moreover, the dual residuals vanish as $\itervar \rightarrow \infty$ such that
\begin{equation}
    \delta_s^{\itervar+1} := \rho(\bar{\pvar}^{\itervar+1} - \bar{\pvar}^{\itervar}) \;\;\to\;\; 0.
\end{equation}
\end{prop}
\begin{proof}
    This result has been shown in~\cite[Appx.~A]{MAL-016}.
\end{proof}

Proposition~\ref{prop:admm_converge} implies that the shared consensus variable has stabilized. For nonconvex problems, ADMM is often observed to converge to a stationary point under certain assumptions, but global optimality is not guaranteed, see~\cite{wang2019_nonconvexadmm} for additional details.

\subsection{Sequential Convex Programming}

Consider the discrete-time dynamics
\begin{align}
    x_{k+1} = f_k(t_k, x_k, u_k),\quad&k\in[K-1],
\end{align}
where $x_k\in\mbb{R}^{n_x}$ is the state, $u_k\in\mbb{R}^{n_u}$ is the control input, and $f_k:\mbb{R}^{n_x}\times\mbb{R}^{n_u} \to \mbb{R}^{n_x}$ are the discrete-time system dynamics between times $t_k$ and $t_k + k\Delta t$. We introduce the variable
\begin{align}
    \pvar_k = \bma x_k^\top & u_k^\top \ema^\top,
\end{align}
where $\pvar_k\in\mbb{R}^{n_\pvar} = \mbb{R}^{n_x + n_u}$. Note that for ease of notation, we let $u_K = \mbf{0}_{n_u}$. Moreover, we let $z = z_{0:K}$, where $z\in\mbb{R}^{n_z\times K}$.
The optimal control problem is
\begin{subequations}\label{eq:noncvx_ocp}
    \begin{alignat}{5}
        &&\min_{\pvar}~&J_K(\pvar_K) + \sum_{k=0}^{K-1} J_k(\pvar_k)&&\\
        &&\st~&e^x \pvar_{k+1} = f_k(\pvar_k),\,&&k\in[K-1],\\
        &&&g(\pvar_k)\leq 0,\,&&k\in[K],\\
        &&&h(\pvar_k) = 0,\,&&k\in[K],\\
        &&&\pvar_k\in\mc{Z}^c_k,\,&&k\in[K],
    \end{alignat}
\end{subequations}
where $J_k:\mbb{R}^{n_\pvar}\to\mbb{R}$ is the running cost function, $J_K:\mbb{R}^{n_\pvar}\to\mbb{R}$ is the terminal cost function, $e^x\in\mbb{R}^{n_\pvar}$ is a row matrix that extracts $x$ from $\pvar$, i.e.~$e^x = \bma \mbf{1}_{n_x}^\top & \mbf{0}_{n_u}^\top \ema$, $\mc{Z}^c_k\subseteq\mbb{R}^{n_\pvar}$ is the set of convex state and control constraints at time step $k$, $g:\mbb{R}^{n_\pvar}\to\mbb{R}^{n_g}$ are nonconvex inequality constraints, and $h:\mbb{R}^{n_\pvar}\to\mbb{R}^{n_h}$ are nonconvex equality constraints. We assume that $J, f_k, g, h\in\mc{C}^1$ are closed, and nonconvex, and that $\mc{Z}^c_k$ for $k\in[K]$ is closed and convex. We note that this problem can be the discretized form of a continuous-time optimal control problem and can include path constraints, see~\cite{Elango_ctcs} for details.

We apply the prox-linear method~\cite{drusvyatskiy_proxlin, szmuk_ptr, reynolds_ptr} to solve~\eqref{eq:noncvx_ocp}. For a general function $\Xi:\mbb{R}^{n_\pvar}\to\mbb{R}^{n_\Xi}$, we denote the linearized function
\begin{align*}
    \tilde{\Xi}(\bar{\pvar}, \pvar) = {\Xi}(\bar{\pvar}) + \nabla{\Xi}^\top(\bar{\pvar}) (\pvar - \bar{\pvar}),
\end{align*}
where $\bar{\pvar}$ is the state about which the function is linearized.
We can formulate~\eqref{eq:noncvx_ocp} as an unconstrained minimization problem by penalizing the nonconvex constraints~\cite{nocedalwright_2006}. The convex constraints are enforced through an indicator function. At the $(\itervar+1)^{\mrm{th}}$ iteration, the nonconvex cost and constraints are linearized about the solution to the $\itervar^{\mrm{th}}$ subproblem. We define the linear function
\begin{multline}\label{eq:lin_cost}
    \Theta(\pvar^{\itervar}, \pvar) \coloneq \tilde{J}_K(\pvar_K^{\itervar}, \pvar_K) + \sum_{k=0}^{K-1}\tilde{J}_{k}(\pvar_k^{\itervar}, \pvar_k) + \mathds{1}_{\mc{Z}^c}(\pvar)\\
        + w^{1}\sum_{k=0}^{K-1}\norm{\tilde{f}_k(\pvar_k^{\itervar}, \pvar_k)}_1
        + w^{2}\mbf{1}_{n_g}^\top \abs{\tilde{g}(\pvar^{\itervar}, \pvar)} \\
        + w^3\norm{\tilde{h}(\pvar^{\itervar}, \pvar)}_1,
\end{multline}
where $w^1, w^2, w^3 \in \mbb{R}_+$ are user-selected weights, and ${\mc{Z}^c = \bigcup_{k\in[K]} \mc{Z}^c_k}$. Note that if the running cost and terminal cost functions are convex, they need not be linearized, and $\tilde{J}_K(\pvar_K^{\itervar}, \pvar_K)$ and $\tilde{J}_k(\pvar_K^{\itervar}, \pvar_K)$ in \eqref{eq:lin_cost} are replaced by ${J}_K(\pvar_K)$ and ${J}_k(\pvar_K)$, respectively.
Since the linearizations of the cost and constraints are only accurate in the neighborhood of the trajectory about which they are performed, deviation from that trajectory is penalized with a trust region, resulting in the function
\begin{align}\label{eq:lin_cost_tr}
    \Gamma(\pvar^{\itervar}, \pvar) \coloneq \Theta(\pvar^{\itervar}, \pvar) + \frac{w^\mrm{p}}{2}\sum_{k=0}^{K}\norm{\pvar_k - \pvar_k^{\itervar}}_2^2,
\end{align}
where $w^\mrm{p}\in\mbb{R}_+$ is a user-selected proximal weight. The SCP framework solves~\eqref{eq:noncvx_ocp} by successively solving the \emph{convex subproblem}
\begin{alignat}{4}\label{eq:cvx_ocp}
    &&z^{\itervar+1} = \argmin_{\pvar}~&\Gamma(\pvar^{\itervar}, \pvar).
\end{alignat}
Between SCP iterations, the trajectory about which the problem is linearized is updated with the solution of the previous iterate. This process is repeated until a convergence criterion is satisfied. Algorithm~\ref{alg:scp} outlines a multi-start version of the standard SCP algorithm, in which $n_a$ initial trajectories are used to initialize the algorithm.
\begin{algorithm}
\caption{Multi-Start Standard SCP}
\label{alg:scp}
\begin{algorithmic}
\Require $\epsilon_c\in\mbb{R}_+,\; n^\mrm{max}\in\mbb{Z}_+,\; \pvar^0_{i} \text{ for } i\in[n_a]$
\For{$i=1,\ldots,n_a$}
\While{$\itervar \leq \itervar^{\mrm{max}}$, $\norm{\Theta (\pvar_{i}^{\itervar-1}, \pvar_{i}^{\itervar})} > \epsilon_c$}
\State $\pvar_{i}^{\itervar+1} \leftarrow \argmin_{\pvar} \Gamma(\pvar_{i}^{\itervar}, \pvar)$
\State $\itervar \leftarrow \itervar + 1$
\EndWhile
\EndFor
\end{algorithmic}
\end{algorithm}

\section{Exploratory SCP Algorithm} \label{sec:algorithm}

The solution found using standard SCP is a stationary point of~\eqref{eq:noncvx_ocp} in the neighborhood of the initial trajectory, $\pvar^0$~\cite{Elango_ctcs, drusvyatskiy_proxlin}. SCP solutions are therefore very sensitive to their initialization. In this section, we propose a method for exploring the nonconvex solution space of a problem using a modified SCP algorithm, referred to as OS-SCP.

\subsection{Operator-Splitting SCP (OS-SCP)}

The trust region term in~\eqref{eq:lin_cost_tr} promotes validity of the linearized model by penalizing deviation from the trajectory about which the problem was linearized. However, this impedes exploration of the nonconvex solution space and can prevent the algorithm from exiting a local minimum and finding a lower-cost solution. We therefore modify the conventional penalized linearization in~\eqref{eq:lin_cost_tr}.

\subsubsection{Primal update}
We use $n_a$ virtual agents to explore the solution space of an optimal control problem through iterative solves of~\eqref{eq:admm}. The OS-SCP algorithm begins by solving~\eqref{eq:ADMM_agent_update} with $p_i(z_i) = \Theta(z_i^j, z_i)$ for each agent, which is analogous to a standard SCP iteration, given by~\eqref{eq:lin_cost_tr}, with a modified penalty term. The subproblem can be expressed as
\begin{align}\label{eq:osscp_primalupdate}
    \pvar_i^{\itervar+1} &= \argmin_{\pvar}~\Theta(z_i^j, z) + \frac{\rho}{2} \norm{\pvar -\bar{\pvar}^\itervar+\dvar_i^\itervar}^2_2,\; &&i\in[n_a]
\end{align}
where $z_i$ is the trajectory of the $i^\mrm{th}$ agent, $z_i^j$ is the solution to the $j^\mrm{th}$ convex subproblem for the $i^\mrm{th}$ agent, and $\rho$ is the consensus penalty parameter. We define the penalized linearized cost as the objective in~\eqref{eq:osscp_primalupdate} so that
\begin{multline}
    \Gamma^c(\pvar^{\itervar}_i, \bar{\pvar}^{\itervar}, \xi_i^j, \pvar) = \Theta(z_i^j, z) + \frac{\rho}{2}\norm{\pvar -\bar{\pvar}^\itervar+\dvar_i^\itervar}^2_2.
\end{multline}
The modified penalty term penalizes deviation of the state of each agent from the consensus state, $\bar{\pvar}$. This promotes the formation of a consensus among the agents.

\subsubsection{Consensus update}
We begin by defining the convexified constraint set
\begin{align}
    \mc{Z} = \mc{Z}^c \cap \mc{Z}^n,
\end{align}
where $\mc{Z}\subseteq\mbb{R}^{n_\pvar\times K}$, and $\mc{Z}^n\subseteq\mbb{R}^{n_\pvar\times K}$ is the set of convexified nonconvex constraints, linearized about the mean trajectory, $\hat{\pvar}^{\itervar} = \frac{1}{n_a}\sum_{i=1}^{n_a} (\pvar_i^{\itervar})$. The set of convexified nonconvex constraints is
\begin{multline}
    \mc{Z}^n = \{\pvar \mid \tilde{g}(\hat{\pvar}_k^\itervar, \pvar_k) \leq 0, \tilde{h}(\hat{\pvar}_k^\itervar, \pvar_k) = 0,\; k \in[K], \\{\pvar}_{k+1} = \tilde{f}_k(\hat{\pvar}_k^\itervar, \pvar_k), \; k\in[K-1] \}.
\end{multline}

We wish to obtain a consensus trajectory that is feasible with respect to the primal problem. We therefore define $q(\bar{z})$  in~\eqref{eq:ADMM_consensus_update} as the indicator function of the convex set $\mc{Z}$, resulting in
\begin{align}\label{eq:osscp_consensusupdate}
    \bar{\pvar}^{\itervar+1}&=\argmin_{\bar{\pvar}} \mathds{1}_\mc{Z}(\bar{\pvar}) + \frac{\rho}{2}~\sum_{i=1}^{n_a}\norm{\pvar^{\itervar+1}_i-\bar{\pvar}+\dvar_i^\itervar}^2_2,\\
    &\text{where} \  \mathds{1}_\mc{Z}(\bar{\pvar}) = 
    \begin{cases}
        0 & \text{if } \bar{\pvar}\in \mc{Z}\\
        +\infty & \text{if } \bar{\pvar} \notin \mc{Z}
    \end{cases}.\nonumber
\end{align}
Equation~\eqref{eq:osscp_consensusupdate} is solved by projecting the mean of the agent's trajectories onto the convexified constraint set, resulting in
\begin{align}
    \bar{\pvar}^{j+1} = \frac{\rho}{2}~
    \Pi_{\mathcal{Z}}\!\left(
        \frac{1}{n_a}\sum_{i=1}^{n_a} (\pvar_i^{\itervar+1} + \xi_i^{j})
    \right),
\end{align}
where $\Pi_{\mathcal{Z}}(\cdot)$ denotes the Euclidean projection operator,
\begin{align}
    \Pi_{\mathcal{Z}}(v) := \argmin_{\bar{\pvar} \in \mathcal{Z}} \norm{\bar{\pvar} - v}_2^2,
\end{align}
and $\dvar_i^{\itervar}$ is the dual variable for agent $i$ at the $\itervar^\mrm{th}$ OS-SCP iteration.

\subsubsection{Dual update}
Finally,~\eqref{eq:ADMM_dual_update} is solved to update the dual variables.

The algorithm is outlined in Algorithm~\ref{alg:os_scp}. The agents begin with different initializations. The iterations outlined above are performed until the primal and dual variables converge below a specified tolerance.

\begin{algorithm}
\caption{Operator-Splitting SCP}
\label{alg:os_scp}
\begin{algorithmic}
\Require $\epsilon_r,\; \epsilon_s, \; \epsilon_c\in\mbb{R}_+,\; j^\mrm{max}\in\mbb{Z}_+,\; \pvar^0_{i} \text{ for } i\in[n_a]$
\While{$\left(\norm{\delta_r^{j}} > \epsilon_r, \; \delta_s^j > \epsilon_s,\;  j \leq j^{\mrm{max}},\; \norm{\Theta (\bar{\pvar}^{\itervar-1}, \bar{\pvar}^{\itervar})} > \epsilon_c\right)$}
\For{$i=1,\ldots,n_a$}
\State $\pvar_{i}^{\itervar+1} \leftarrow \argmin_{\pvar} \Gamma^c (\pvar_{i}^{\itervar}, \bar{\pvar}_{i}^{\itervar}, \xi_i^j, \pvar)$
\EndFor
\State $\bar{\pvar}^{\itervar+1} \leftarrow \frac{\rho}{2}~\Pi_{\mathcal{Z}}\!\left(\frac{1}{n_a} \sum_{i=1}^{n_a} \br{\pvar_i^{\itervar+1} + \dvar_i^{\itervar}}\right)$
\State $\dvar_i^{\itervar+1} \leftarrow \dvar_i^{\itervar} + \br{\pvar_i^{\itervar+1} - \tilde{\pvar}^{\itervar+1}}$
\State $\itervar \leftarrow \itervar + 1$
\EndWhile
\end{algorithmic}
\end{algorithm}

While we describe the method using consensus across all solution variables for ease of notation, the algorithm does not rely on this assumption. Consensus can be restricted to a subset of the state variables, which may be preferable when only certain portions of the state space pose challenges with local minima and consensus elsewhere is unnecessary.

\section{Numerical Results} \label{sec:results}
We now demonstrate the performance of the exploration-focused operator-splitting SCP method against standard SCP in two illustrative examples: a simple obstacle avoidance example, and an obstacle avoidance example over a landscape with a non-uniform cost field.

\subsection{Unicycle Trajectory Optimization}
Consider the discrete-time kinematic model for a unicycle with constant velocity
\begin{equation}
    x_{k+1} = f_k(t_k,{x}_k,u_k) = \begin{bmatrix}
        r^x_k+v\cos \theta_k \Delta t \\
        r^y_k+v\sin \theta_k \Delta t \\
        \theta_k + u_k \Delta t
    \end{bmatrix},
\end{equation}
where the state is $x_k = [r^x_k \quad r^y_k \quad \theta_k]^\top \in \mathbb{R}^3$ and control $u_k \in \mathbb{R}$ is the yaw rate. The speed is defined as a constant $v > 0$, and the dynamics are discretized with the uniform time step $\Delta t$. We aim to find an optimal trajectory and nominal control to bring the vehicle from an initial state to a desired terminal state, while avoiding obstacles. The problem setup is shown in Figure~\ref{fig:uni_landscape}.
\begin{figure}
    \centering
    \includegraphics[width=\linewidth]{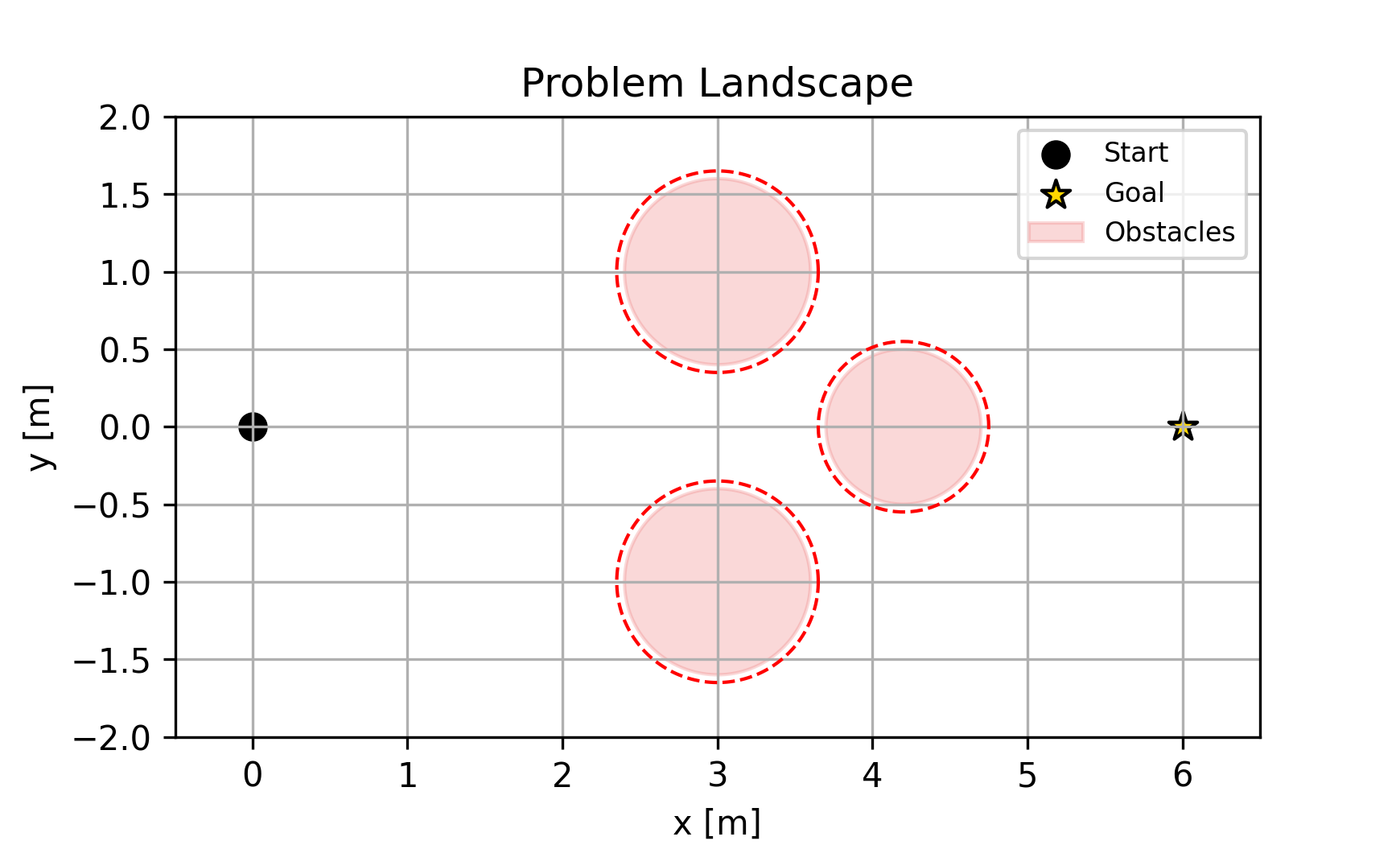}
    \caption{Problem setup for simple obstacle avoidance example.}
    \label{fig:uni_landscape}
\end{figure}

Let ${x}_g$ be the desired terminal state. We express the problem in the form of \eqref{eq:noncvx_ocp},
where the convex terminal and running cost functions are
\begin{subequations}\label{eq:uni_cost}
\begin{align}
    J_K(\pvar_K) &= (e^x \pvar_K - {x}_g)^\top Q_{g} (e^x \pvar_K - {x}_g), \\ 
    J_k(\pvar_k) &= \norm{e^u z_k}^2_2 ,
\end{align}
\end{subequations}
with $Q_{g} \succeq 0$. The nonconvex inequality constraints are 
\begin{align}
    g_\iota(\pvar_k) = R_\iota - \norm{e_r \pvar_k - c_\iota} \leq 0, \; \iota\in\{1, 2, 3\},
\end{align}
where $R_\iota \in \mbb{R}_+$ is the radius of obstacle $\iota$ and $c_\iota\in\mbb{R}^2$ is the center of obstacle $\iota$.

Following the prox-linear methodology, at iteration $\itervar$, we formulate the nonconvex optimization problem as a convex unconstrained minimization problem by penalizing the constraints, and linearizing the cost and constraints about the solution to the $(\itervar-1)^\mrm{th}$ subproblem. We formulate the penalized cost as in~\eqref{eq:lin_cost_tr} with linearized cost from~\eqref{eq:lin_cost}. Note that since the terminal and running costs are convex, they need not be convexified, and $\tilde{J}_K(\pvar_K^j, \pvar_K)$ and $\tilde{J}_k(\pvar_k^j, \pvar_k)$ in~\eqref{eq:lin_cost} are replaced by ${J}_K(\pvar_K)$ and ${J}_k(\pvar_k)$, respectively. 

\subsubsection{SCP solution}
We first solve this problem with the standard SCP method according to Algorithm~\ref{alg:scp} with three initial guesses. The first guess is an arc veering toward the left of the vehicle, the second is a straight line from initial state to the goal, and third is an arc veering toward the right of the vehicle. The initial guesses and the converged trajectories are shown in Figure~\ref{fig:uni_single_agent}, and the corresponding final costs and total iterations for each solve are shown in Table~\ref{table:uni_scp}. From Figure~\ref{fig:uni_single_agent}, it is evident that the SCP solution for each initialization converges to a solution in the neighborhood of the corresponding initial guess. This demonstrates the algorithm's difficulty in escaping the stationary point near its initialization. Additionally, Table~\ref{table:uni_scp} shows that only the solution with a straight initial guess converges to a minimum cost solution. This dependence on a good initial guess is a known limitation of SCP. In this visually intuitive example, a straight line guess might be an obvious choice. However, generating good (and especially dynamically feasible) initial guesses is generally not trivial. From this example, the need for an exploration focused, SCP-based algorithm is clearly evident.

\begin{figure}
    \centering
    \includegraphics[width=\linewidth]{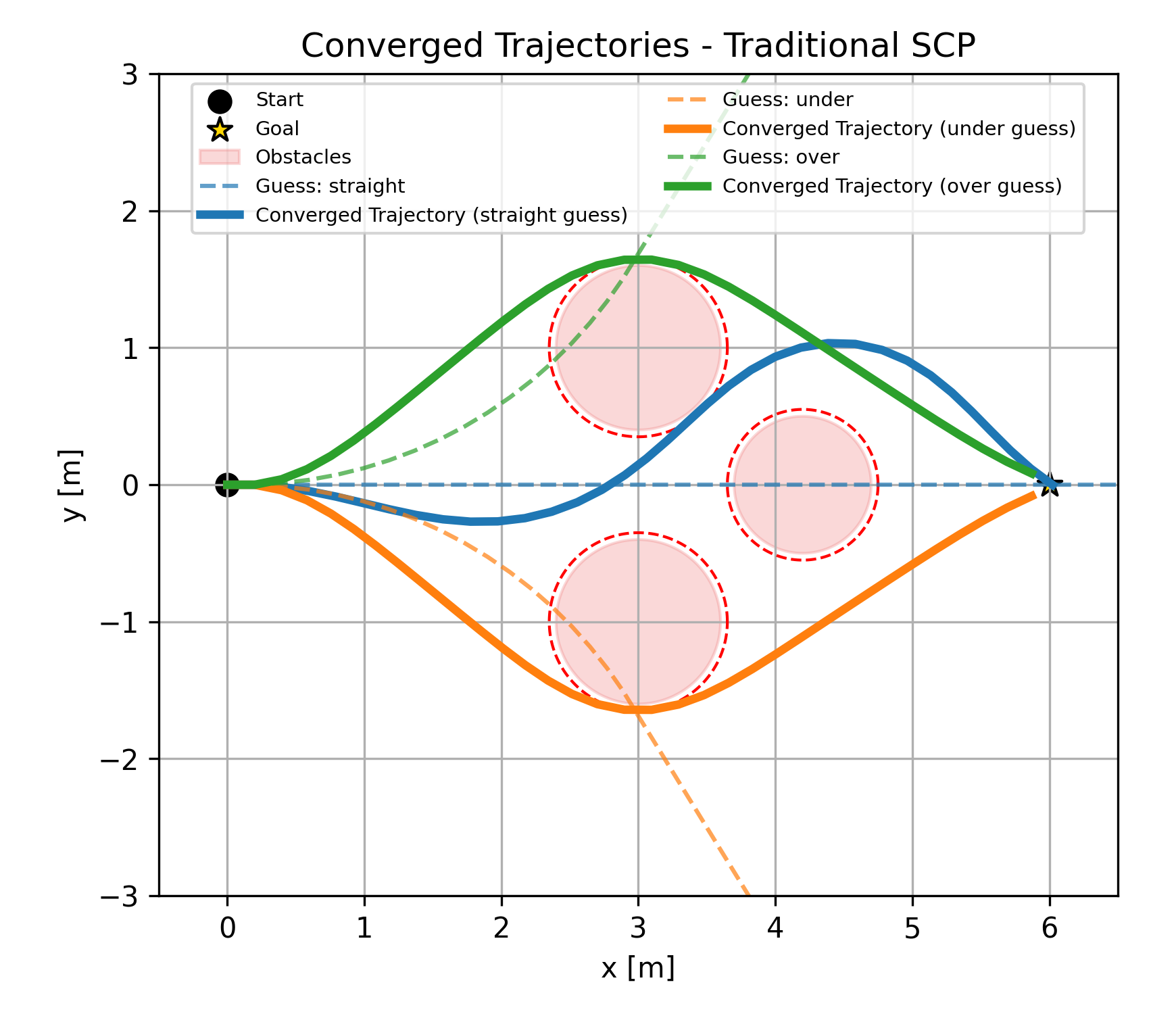}
    \caption{Initial and converged trajectories computed using standard SCP algorithm.}
    \label{fig:uni_single_agent}
\end{figure}

\begin{table}
\caption{SCP Results}
\centering
\begin{tabular}{ p{2cm} p{1.5cm}p{1.5cm}  }
 \toprule
 Guess Direction & Cost & Iterations\\
 \midrule
 Over  & 0.230 & 27 \\
 Straight   & 0.155 & 36 \\
 Under &   0.230  & 27\\
 \bottomrule
\end{tabular}
\label{table:uni_scp}
\end{table}

\subsubsection{OS-SCP solution}
We now solve the problem using the proposed OS-SCP method, summarized in Algorithm~\ref{alg:os_scp}. At each iteration, each agent solves a convex subproblem. A consensus update is then performed, and finally a dual update is performed. This process repeats until the primal and dual residuals fall below specified tolerances $\epsilon_r\in\mbb{R}_+$ and  $\epsilon_s\in\mbb{R}_+$, respectively. The same three initial guesses as before are used to initialize the three agents, where agent 1 is assigned the ``upper" guess, agent 2 the ``straight" guess, and agent 3 the ``lower" guess. The converged solution is shown in Figure~\ref{fig:uni_admm}, and the evolution of the primal and dual residuals are shown in Figure~\ref{fig:p_d_residuals}. The residuals converge to zero over iterations, and OS-SCP successfully forms a consensus between all three agents. The consensus converges to the same trajectory as the lowest cost solution of the standard SCP example, demonstrating the ability for the OS-SCP method to pull agents out of the local minima near to their initialization. We compare the performance of both methods in Table~\ref{table:uni_admm}, where OS-SCP finds an equal cost trajectory in fewer iterations than standard SCP. Since at each iteration of OS-SCP, three convex subproblems are solved (one per agent), and three standard SCP problems are solved (one per initial guess), the run times of the methods are approximately equal. 

\begin{figure}
    \centering
    \includegraphics[width=\linewidth]{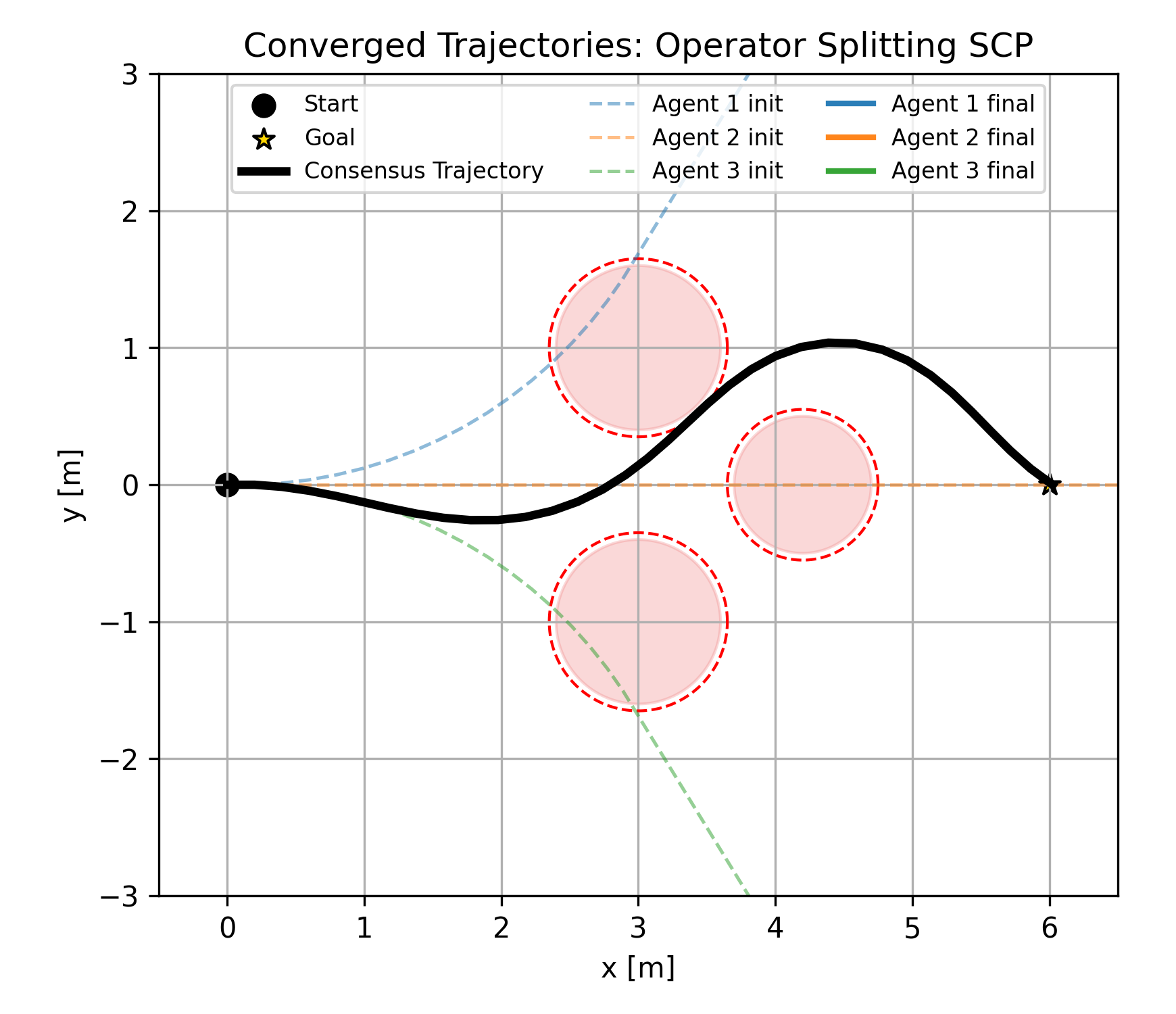}
    \caption{Initial and converged trajectories computed using OS-SCP algorithm.}
    \label{fig:uni_admm}
\end{figure}

\begin{figure}
    \centering
    \includegraphics[width=\linewidth]{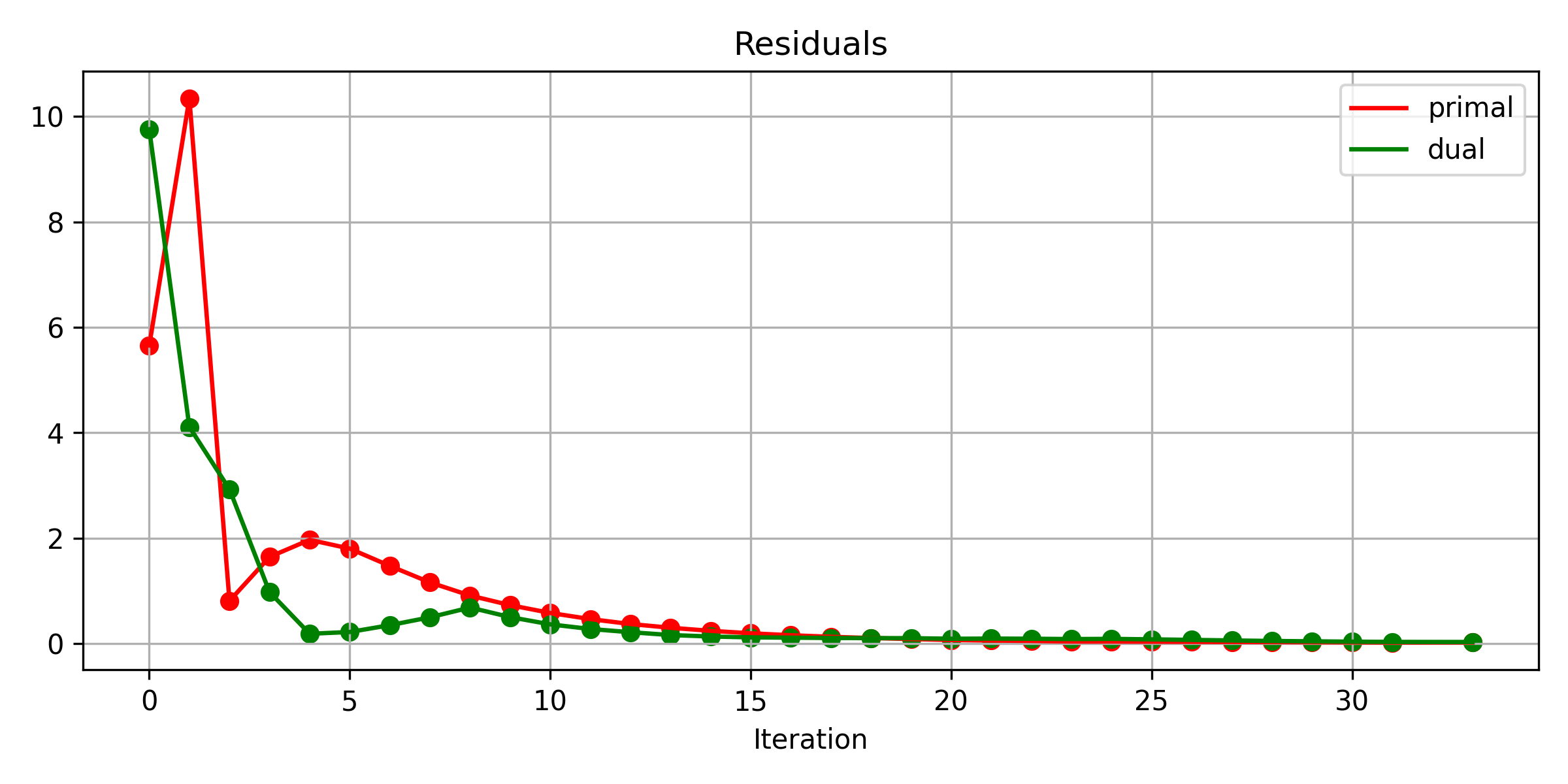}
    \caption{Primal and dual residuals over iterations for OS-SCP.}
    \label{fig:p_d_residuals}
\end{figure}

\begin{table}
\centering
\caption{Numerical Comparison - Unicycle Trajectory}
\begin{tabular}{ p{2cm} p{1.5cm} p{1.5cm}  }
 \toprule
 Method & Cost & Iterations\\
 \midrule 
 Standard SCP  & 0.155 & 36 \\
 OS-SCP  & 0.155 & 33 \\
 \bottomrule
\end{tabular}
\label{table:uni_admm}
\end{table}

\subsection{Unicycle Trajectory Optimization with Gaussian Terrain Fields}
The second problem extends the first example. Consider the same kinematic model and problem setup. However, we add spacial preference biases via a Gaussian terrain field. The cost map created by the terrain field is expressed as
\begin{align}\label{eq:terrain_cost}
J_{\rm map}(r) \;=\; \sum_{\ell=1}^L \exp\!\Big(-\tfrac12 (r-\mu_\ell)^\top \Sigma_\ell^{-1}(r-\mu_\ell)\Big),
\end{align}
where $\mu_\ell \in \mathbb{R}^2$ is the position of the center of the $\ell$-th Gaussian field, and $\Sigma_\ell \succeq 0$ controls the shape of the field and its amplitude. At the $j^\mrm{th}$ iteration, the cost map in \eqref{eq:terrain_cost} is linearized about the solution to the $(j-1)^\mrm{th}$ subproblem, $\pvar^{\itervar}$, and added to the linearized cost function in \eqref{eq:uni_cost}. Then, SCP solves the same unconstrained convex subproblem as in \eqref{eq:cvx_ocp}, but with updated $\Gamma(\pvar^{\itervar}, \pvar)$ to include the Gaussian terrain field cost.

The new problem setup is shown in Figure~\ref{fig:uni_landscape_gauss}, where Gaussian cost fields are added between the upper and lower corridors of the obstacles. The terrain incentivizes traveling through the lower corridor by giving it negative cost, and penalizes the upper corridor with a higher cost terrain. 

\begin{figure}
    \centering
    \includegraphics[width=1.0\linewidth]{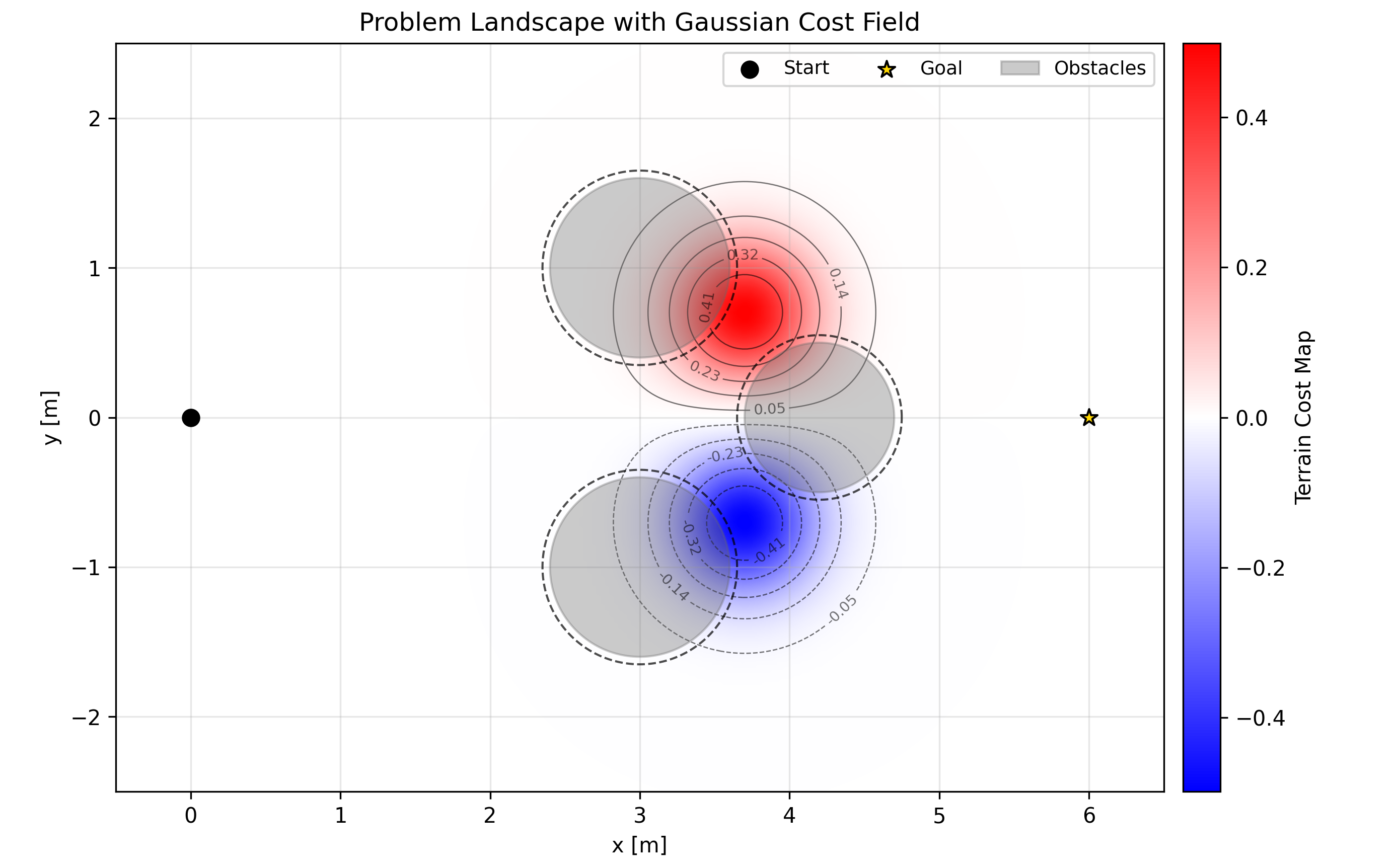}
    \caption{Problem setup for trajectory optimization problem with Gaussian terrain field.}
    \label{fig:uni_landscape_gauss}
\end{figure}

\subsubsection{SCP Solution}
As with the previous example, we first solve the problem with the standard SCP method outlined in Algorithm~\ref{alg:scp}. We initialize three independent standard SCP solves with the same three initial guesses as from the first example. Figure~\ref{fig:uni_scp_gauss_og} shows that each initialization results in a converged trajectory in the neighborhood of the initial guess. Notably, the ``straight" guess gets trapped in a local minimum in the upper corridor. In order for this method to find a minimum cost solution, we need to add a new initial guess close to the global optimum which passes directly through the lower corridor. This solution is depicted in Figure~\ref{fig:uni_scp_gauss}. The only initialization to converge to the lowest cost trajectory is the one whose initial guess passes through the optimal region. This further demonstrates the sensitivity of the standard SCP algorithm to the initial guess. Again, for this visually intuitive example choosing an initial guess near the global optimum is possible. However, for systems in higher dimensions with more complex solution spaces, an initial guess close to the global optimum is generally difficult to generate.

\subsubsection{OS-SCP Solution}
We then solve the problem with OS-SCP, where we use 3 agents with the same initial guesses as in the original example: ``straight", ``lower", and ``upper". The agents explore the solution space before forming a consensus through the lower corridor, successfully finding the lowest cost solution, as shown in Figure~\ref{fig:uni_admm_gauss}, without the need for an initial guess that passes through the lower corridor.  Table~\ref{table:uni_gauss_compare} compares the OS-SCP method against the lowest cost solution from standard SCP. The OS-SCP consensus converges to the same trajectory as the lowest cost solution from the standard SCP solves in near the same number of iterations, but does not have the same need for an initial guess near the global optimum. This example demonstrates the exploratory nature of the OS-SCP method and its lack of dependency on an accurate initial guess. Compared to the standard SCP approach, OS-SCP also provides the benefit of removing the human-in-the-loop requirement of selecting the best trajectory from a set of converged trajectories in the case of multiple unique trajectories with equivalent costs. 

\begin{figure}
    \centering
    \includegraphics[width=1.1\linewidth]{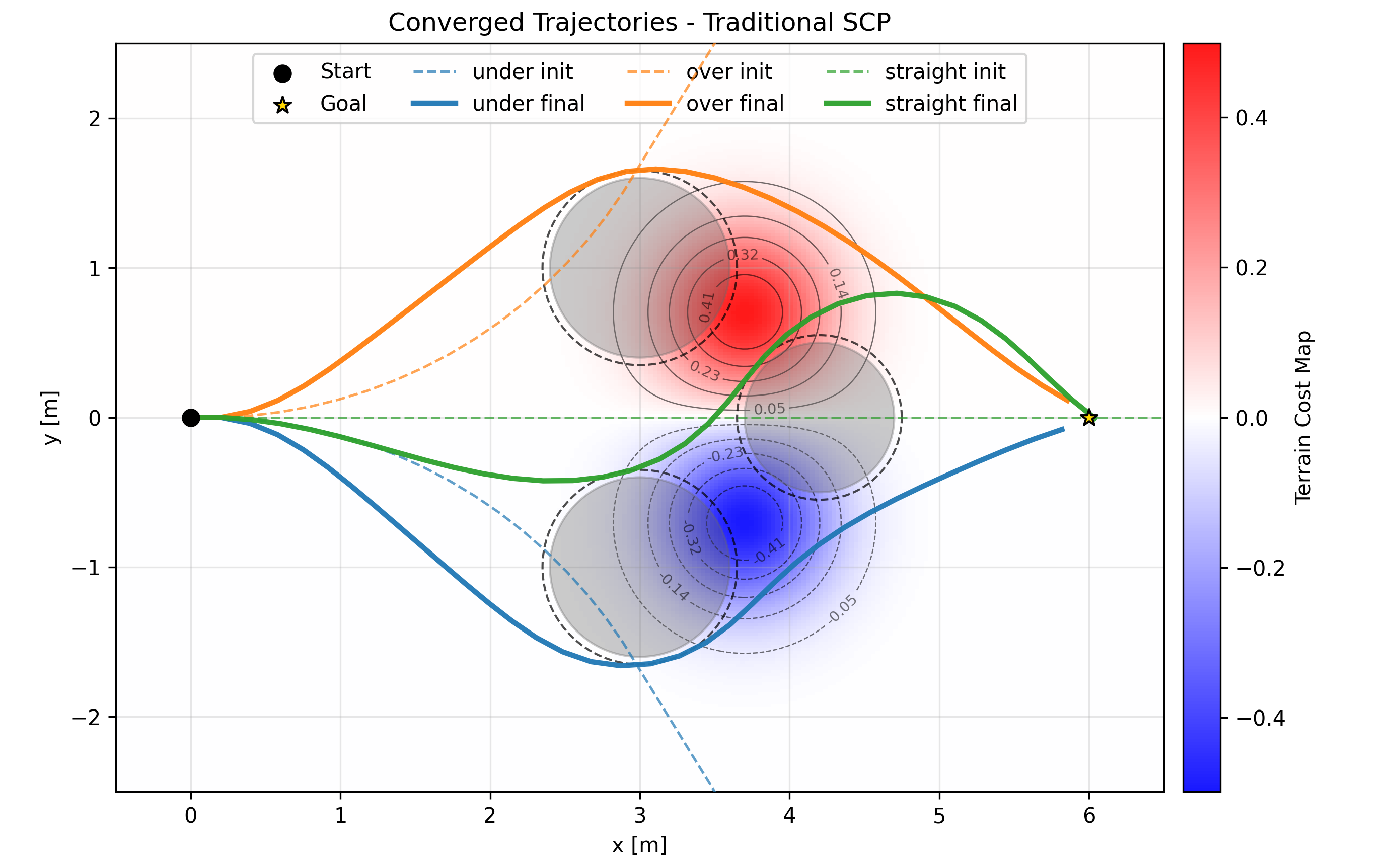}
    \caption{Initial and converged trajectories with Gaussian terrain field cost computed using standard SCP algorithm.}
    \label{fig:uni_scp_gauss_og}
\end{figure}

\begin{figure}
    \centering
    \includegraphics[width=1.1\linewidth]{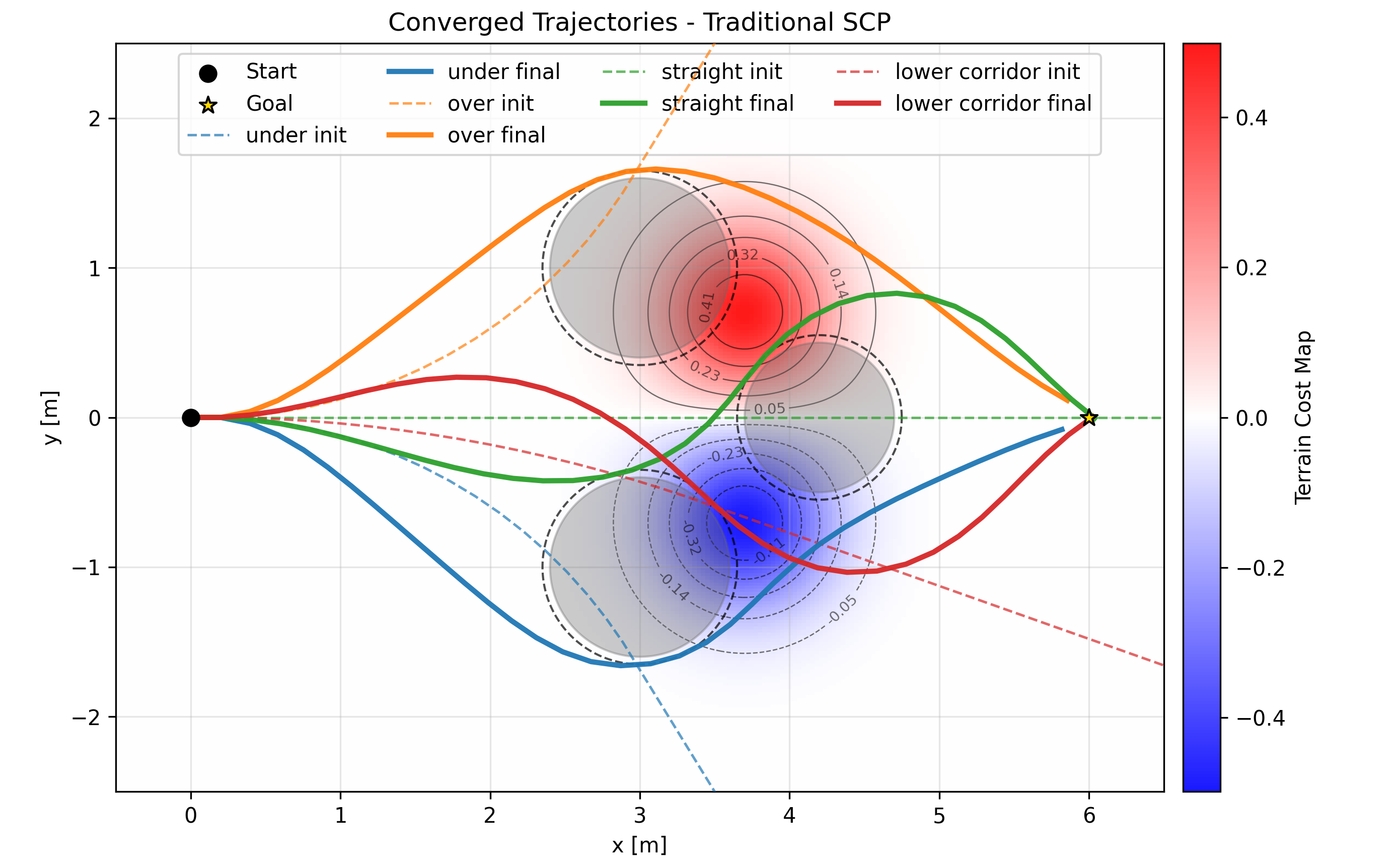}
    \caption{Initial and converged trajectories with Gaussian terrain field cost computed using standard SCP algorithm, including additional ``lower corridor" guess.}
    \label{fig:uni_scp_gauss}
\end{figure}

\begin{figure}
    \centering
    \includegraphics[width=1.1\linewidth]{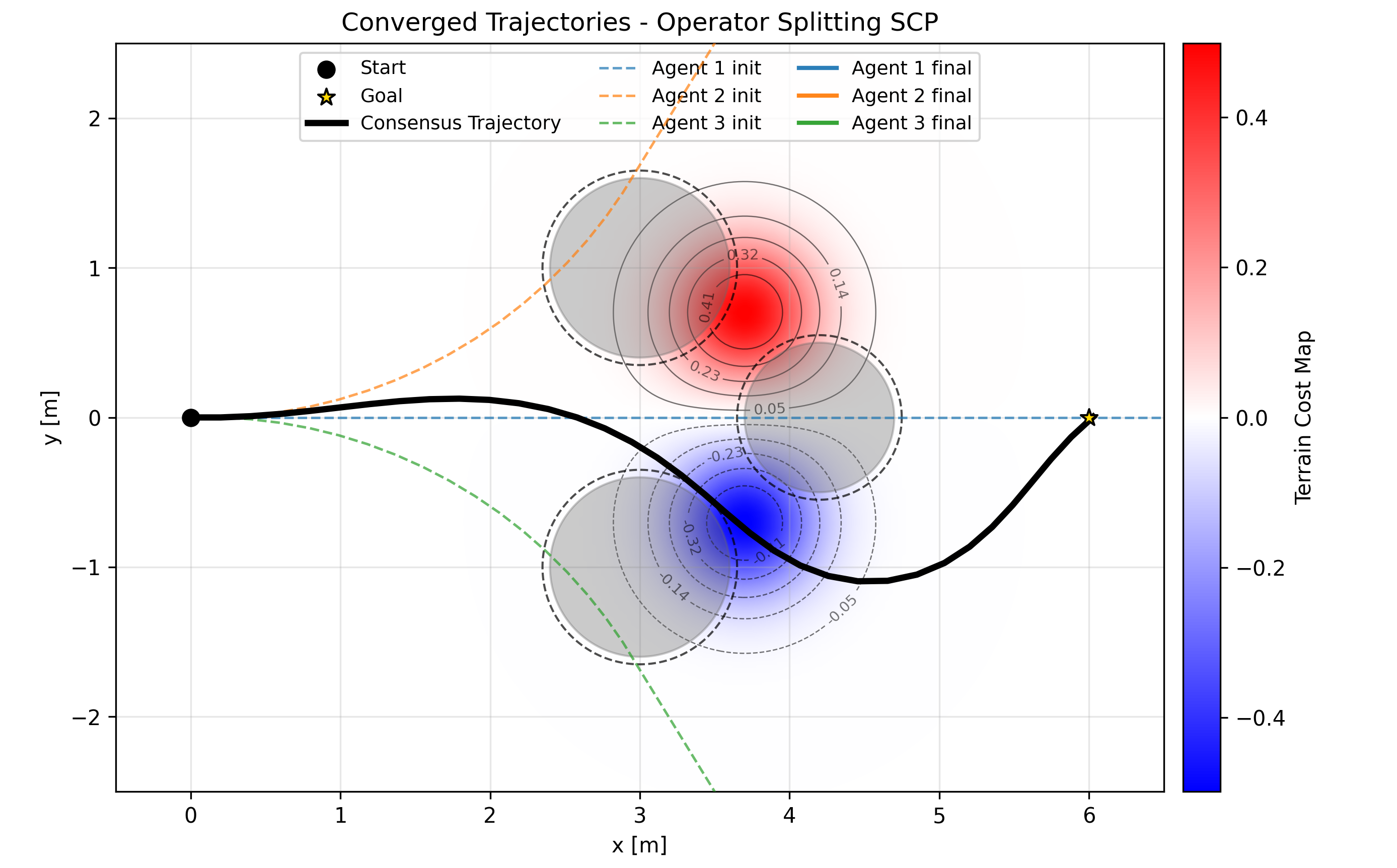}
    \caption{Initial and converged trajectories with Gaussian terrain field cost computed using OS-SCP algorithm.}
    \label{fig:uni_admm_gauss}
\end{figure}

\begin{table}
\centering
\caption{Numerical Results - Unicycle Trajectory with Gaussian Terrain Field}
\begin{tabular}{ p{3cm} p{1.5cm} p{1.5cm}  }
 \toprule
 Method & Cost & Iterations\\
 \midrule 
 Standard SCP (3 guesses) & -0.136 & 28 \\
 Standard SCP (4 guesses) & -0.715 & 31 \\
 OS-SCP  & -0.715 & 33 \\
 \bottomrule
\end{tabular}
\label{table:uni_gauss_compare}
\end{table}

\section{Conclusion} \label{sec:conclusion}

This paper introduces the OS-SCP framework that promotes exploration in nonconvex trajectory optimization problems while preserving the feasibility and structure of standard SCP. Rather than solving a single locally convexified subproblem from one initialization, OS-SCP instantiates multiple agents with diverse initial guesses and couples them through a consensus ADMM update. This mechanism allows the population of agents to search not only the local minima about each initial guess, but also between local minima in the cost landscape. Results show that this can lead to OS-SCP finding lower-cost solutions than the standard SCP algorithm initialized with the same set of initial guesses.


\bibliographystyle{IEEEtran}

\bibliography{references}

\end{document}